\documentclass[a4paper,10pt]{amsart}

\input xy
\xyoption{all}

\usepackage{geometry}
\usepackage{amsmath}
\usepackage{amssymb}
\usepackage{cite}
\usepackage{amsthm}  
\usepackage{mathtools}
\usepackage{amsmath, amsthm, amssymb}
\usepackage{listings}
\usepackage{url}
\usepackage{hyperref}
\usepackage[utf8]{inputenc}
\usepackage{graphicx}
\usepackage{amsfonts}
\usepackage{tikz-cd}

\newcommand{\K}{\mathbf{K}}
\newcommand{\LS}{\operatorname{LS}}
\newcommand{\lea}{\leap{\K}}
\newcommand{\Ii}{\mathbb{I}}

\theoremstyle{definition}
\newtheorem{mydef}{Definition}[section]
\newtheorem{lem}[mydef]{Lemma}
\newtheorem{thm}[mydef]{Theorem}
\newtheorem{cor}[mydef]{Corollary}

\newtheorem{question}[mydef]{Question}

\newtheorem{prop}[mydef]{Proposition}

\newtheorem{remark}[mydef]{Remark}
\newtheorem{notation}[mydef]{Notation}
\newtheorem{fact}[mydef]{Fact}
\newtheorem*{maint}{Theorem 3.14}

\newcommand{\leap}[1]{\le_{#1}}

\newcommand{\gS}{\mathbf{S}}

\newcommand{\cf}{\text{cf }}

\newcommand{\rest}{\upharpoonright}

\newcommand{\concat}{%
  \mathord{
    \mathchoice
    {\raisebox{1ex}{\scalebox{.7}{$\frown$}}}
    {\raisebox{1ex}{\scalebox{.7}{$\frown$}}}
    {\raisebox{.5ex}{\scalebox{.5}{$\frown$}}}
    {\raisebox{.5ex}{\scalebox{.5}{$\frown$}}}
  }
}

\newbox\noforkbox \newdimen\forklinewidth
\forklinewidth=0.3pt \setbox0\hbox{$\textstyle\smile$}
\setbox1\hbox to \wd0{\hfil\vrule width \forklinewidth depth-2pt
 height 10pt \hfil}
\wd1=0 cm \setbox\noforkbox\hbox{\lower 2pt\box1\lower
2pt\box0\relax}
\def\unionstick{\mathop{\copy\noforkbox}\limits}

\setbox0\hbox{$\textstyle\smile$}
\setbox1\hbox to \wd0{\hfil{\sl /\/}\hfil} \setbox2\hbox to
\wd0{\hfil\vrule height 10pt depth -2pt width
               \forklinewidth\hfil}
\newbox\doesforkbox
\setbox\doesforkbox\hbox{\lower 0pt\box1 \lower
2pt\box2\lower2pt\box0\relax}

\def\1nf{\unionstick^{(1)}}

\def\2nf{\unionstick^{(2)}}

\def\3nf{\unionstick^{(3)}}

\def\forkindep{\mathrel{\raise0.2ex\hbox{\ooalign{\hidewidth$\vert$\hidewidth\cr\raise-0.9ex\hbox{$\smile$}}}}}

\newcommand{\type}{\mathbf{gtp}}
\newcommand{\gtp}{\mathbf{gtp}}
\title{Building models in small cardinals in local abstract elementary classes}

\date{\today} 

\author{Marcos Mazari-Armida}
\email{marcos\_mazari@baylor.edu}
\urladdr{https://sites.baylor.edu/marcos\_mazari/}
\address{Department of Mathematics \\ Baylor University \\ Waco, Texas, USA}
\thanks{The first author's research was partially supported by an AMS-Simons Travel Grant 2022--2024.}

\author{Wentao Yang}
\email{wentaoyang@cmu.edu}
\urladdr{https://math.cmu.edu/~wentaoya/}
\address{Department of Mathematical Sciences \\ Carnegie Mellon University \\ Pittsburgh, Pennsylvania, USA}

\begin{document}

\begin{abstract} 

There are many results in the literature where superstablity-like independence notions, without any categoricity assumptions, have been used to show the existence of larger models. In this paper we show that \emph{stability} is enough to construct larger models  for small cardinals assuming a mild locality condition for Galois types.

\begin{thm}
Suppose $\lambda<2^{\aleph_0}$. Let $\K$ be an abstract elementary class with $\lambda \geq \LS(\K)$.  Assume $\K$ has amalgamation in $\lambda$, no maximal model in $\lambda$, and is stable in $\lambda$. If $\K$ is $(<\lambda^+, \lambda)$-local, then $\K$ has a model of cardinality $\lambda^{++}$.
\end{thm}

The set theoretic assumption that $\lambda<2^{\aleph_0}$ and model theoretic assumption of  stability in $\lambda$ can be weakened to the model theoretic assumptions that  $|\gS^{na}(M)|< 2^{\aleph_0}$ for every $M  \in \K_\lambda$ and stability for $\lambda$-algebraic types in $\lambda$. This is a significant improvement of Theorem 0.1.,  as the result holds on some unstable abstract elementary classes.

\end{abstract}

\maketitle

\section{Introduction}

Abstract elementary classes (AECs for short) were introduced by Shelah \cite{Sh:88} to study classes of structures axiomatized in several infinitary logics. Given an abstract elementary class $\K$ and $\lambda$ an infinite cardinal, $\Ii(\K, \lambda)$ denotes the number of non-isomorphic models in $\K$ of cardinality $\lambda$. One of the main test questions in the development of abstract elementary classes is Grossberg's question \cite[Problem (5), p. 34]{sh576}\footnote{Two earlier and weaker versions of this question are Question 21 of \cite{friedman} (due to Baldwin) and Question 4 on page 421 of \cite{Sh:88} (due to Grossberg). The former has a positive answer due to Shelah \cite{sh48}.}:

\begin{question}
Let $\K$ be an AEC and $\lambda\geq \LS(\K)$ be an infinite cardinal. If $\Ii(\K, \lambda)= 1$ and $1 \leq \Ii(\K, \lambda^+) < 2^{\lambda^+}$, must $\K$ have a model  of cardinality $\lambda^{++}$?
\end{question}

This question is still open. When the AEC has a countable  Löwenheim-Skolem-Tarski number and $\lambda = \aleph_0$, the issue of constructing a model in $\aleph_2$ was discussed in Shelah's pioneering papers \cite{sh74}, \cite{sh87a} and \cite[3.7]{Sh:88}. When the AEC has an uncountable Löwenheim-Skolem-Tarski number or $\lambda > \aleph_0$ Grossberg's question is known to be significantly harder.  For instance, Shelah's best approximation to Grossberg's question \cite[\S VI.0.(2)]{shelahaecbook}, which is a revised version of \cite{sh576} published two decades earlier, is over 300 pages long. Additional approximations to Grossberg's question include: \cite[\S II.4.13.3]{shelahaecbook}, \cite[3.1.9]{jrsh875}, \cite[8.9]{vaseya}, \cite[12.1]{vaseyt}, \cite[5.8]{shvas-apal},  \cite[3.3, 4.4]{mv}, \cite[4.2]{m1}, \cite[1.6, 3.7, 5.4]{vaseys}, \cite[4.9]{leu}.  

A key intermediate step to answer Grossberg's question has been to show that stability  and even the existence of a superstablity-like independence notion follow from categoricity in several cardinals. Recently, there are many results where superstablity-like independence notions, without any categoricity assumptions, have been used to show the existence of larger models  \cite[\S II.4.13.3]{shelahaecbook}, \cite[3.1.9]{jrsh875},  \cite[8.9]{vaseya}, \cite[4.2]{m1}.  In this paper, we show that stability, without  any categoricity assumptions, is enough to construct larger models  for small cardinals assuming a mild locality condition for Galois types.



\begin{maint}
Suppose $\lambda<2^{\aleph_0}$. Let $\K$ be an abstract elementary class with $\lambda \geq \LS(\K)$.  Assume $\K$ has amalgamation in $\lambda$, no maximal model in $\lambda$, and is stable in $\lambda$. If $\K$ is $(<\lambda^+, \lambda)$-local, then $\K$ has a model of cardinality $\lambda^{++}$.
\end{maint}
To help us compare our results with previous results, let us recall the following three frameworks: universal classes  \cite{tarski}, \cite{sh300}, tame AECs\cite{tamenessone}  and local AECs \cite{sh576}, \cite{bales}. The first is a semantic assumption on the AEC while the other two are locality assumptions on Galois types (see Definition \ref{tame} and Definition \ref{local}). The relation between these frameworks is as follows: universal classes are $(<\aleph_0, \lambda)$-tame for every $\lambda \geq \LS(\K)$ \cite[3.7]{vaseyd} and $(<\aleph_0, \lambda)$-tame AECs are $(<\lambda^+, \lambda)$-local for every $\lambda \geq \LS(\K)$. The first inclusion is proper and the second inclusion is not known to be proper (see Question \ref{q1}). 

When the AEC has an uncountable Löwenheim-Skolem-Tarski number or $\lambda  > \aleph_0$, Theorem  \ref{main-weak2} is new even for universal classes.  When the AEC has a countable  Löwenheim-Skolem-Tarski number and $\lambda = \aleph_0$, Theorem \ref{main-weak2}    is new for $(\aleph_0, \aleph_0)$-local AECs. For $(<\aleph_0, \aleph_0)$-tame AECs Theorem \ref{main-weak2} can be obtained using \cite[4.7]{shvas-apal}, \cite[5.8]{shvas-apal}, \cite[II.4.13]{shelahaecbook}\footnote{We were unaware of this argument until Sebastien Vasey pointed it when we showed him a final draft of the paper.}, but the result has never been stated in the literature.

A  result similar to Theorem \ref{main-weak2}  is \cite[12.1]{vaseyt}. The main difference is  that Vasey's result has the additional assumption that the AEC is categorical in $\lambda$.  Moreover, Vasey assumes tameness while we only assume the weaker property of locality for Galois types. It is worth mentioning that Vasey does not assume that $\lambda< 2^{\aleph_0}$,  but this is a weak assumption as long as $\lambda$ is a \emph{small} cardinal.



The main difference between the proof of Theorem \ref{main-weak2}  and the previous results is that we focus on finding \emph{one} good type instead of a \emph{family} of good types. A good type in this paper is a $\lambda$-unique type (see Definition \ref{duni}). Once we have this good type,  we carefully build a chain of types above this type to show that every model of cardinality $\lambda^{+}$ has a proper extension and hence show the existence of a model of cardinality $\lambda^{++}$. 

The set theoretic assumption that $\lambda<2^{\aleph_0}$ and model theoretic assumption of  stability in $\lambda$ can be weakened to the model theoretic assumptions that  $|\gS^{na}(M)|< 2^{\aleph_0}$ for every $M  \in \K_\lambda$ and stability for $\lambda$-algebraic types in $\lambda$ (see Theorem \ref{main-2-weak}). The assumption of stability for $\lambda$-algebraic types in $\lambda$ is strictly weaker than stability in $\lambda$ as for instance all elementary classes are stable for $\lambda$-algebraic types in $\lambda$. Moreover, any AEC with disjoint amalgamation in $\lambda$ is stable for $\lambda$-algebraic types in $\lambda$ (see Proposition \ref{p-disj}). Due to this, Theorem \ref{main-2-weak} is a significant improvement to Theorem \ref{main-weak2} as the result holds for some unstable AECs. Theorem \ref{main-2-weak} is new even for universal classes with countable Löwenheim-Skolem-Tarski number and $\lambda = \aleph_0$.

The first version of this paper had an additional section where we presented a positive answer to Grossberg's question  for small cardinals assuming a mild locality condition for Galois types  and without any stability assumptions.  That result relies on a result of Shelah \cite[VI.2.11.(2)]{shelahaecbook} for which Shelah does not provide an argument, for which the \emph{standard} argument does not seem to work, and which we were unable to verify. Due to the status of Shelah's result, following the referee's advice, we decided not to include that result in the paper, but it can be consulted in \cite[4.11]{mawa} and \cite{wyt}.

The paper is organized as follows. Section 2 presents necessary background. Section 3 has the main results. It is worth mentioning that Shelah's papers are labelled using Shelah's numbering instead of their publication year. 

This paper was written while the second author was working on a Ph.D. thesis under the direction of Rami Grossberg at Carnegie Mellon University, and the second author would like to thank Professor Grossberg for his guidance and assistance in his research in general and in this work
specifically. We would like to thank Rami Grossberg for suggesting us to pursue this project and for comments that helped improve the paper. We would also like to thank Sebastien Vasey for many helpful comments that helped improved the paper and for his comments regarding Remark \ref{re<}. We are grateful to the referee for many comments that significantly improved the presentation of the paper, for the equivalence between (2) and (4) of Lemma \ref{many-r} and for Question \ref{q-r}.



\section{Preliminaries}

We assume the reader has some familiarity with abstract elementary classes as presented in \cite[\S 4 - 8]{baldwinbook09}, \cite{grossberg2002} or \cite[\S 2]{shelahaecbook}, but we recall the main notions used in this paper. 

An AEC is a pair $\K=(K \lea)$ where $K$ is a class of structures in a fixed language and $\lea$ is a partial order on $K$ extending the substructure relation such that $\K$ is closed under isomorphisms and satisfies the  coherence property, the L\"{o}wenheim-Skolem-Tarski axiom and the Tarski-Vaught axioms. The reader can consult the definition in \cite[4.1]{baldwinbook09}. 

\begin{notation}
For any structure $M$, we denote its universe by $|M|$, and its cardinality by $\|M\|$. For a cardinal $\lambda$, we let $\K_{\lambda}= \{M\in \K : \|M\| = \lambda \}$. When we write $M \lea N$ we assume that $M, N \in \K$. 
\end{notation}

 For an AEC $\K$, $\K$ has the amalgamation property if for every $M_0\lea M_l$ for $\ell=1,2$, there is $N\in \K$ and $\K$-embeddings $f_\ell:M_\ell\to N$ for $\ell=1,2$ such that $f_1 \restriction_{M_0}=f_2 \restriction_{M_0}$; and $\K$ has no maximal models if every $M\in \K$ has a proper $\lea$-extension in $\K$. For a property $P$, we say that $\K$ has $P$ in $\lambda$ if $\K_\lambda$ has the property $P$.

 Throughout the rest of this section $\K$ is always an abstract elementary class and $\lambda$ is always a cardinal greater than or equal to the Löwenheim-Skolem-Tarski number of $\K$.

We recall the notion of a Galois type. These were originally introduced by Shelah.

\begin{mydef}\
\begin{enumerate}
    \item  For $(b_1, A_1, N_1), (b_2, A_2, N_2)$ such that $N_\ell \in \K$, $A_\ell \subseteq |N|$ and $b_\ell \in N_\ell$ for $\ell =1, 2$, $(b_1, A_1, N_1)E_{\text{at}} (b_2, A_2, N_2)$ if $A
:= A_1 = A_2$, and there exist $\K$-embeddings $f_\ell : N_\ell \to N$ for $\ell = 1, 2$ such that
$f_1 (b_1) = f_2 (b_2)$ and $f_1\restriction_{A}=f_2\restriction_A$. Let $E$ be the transitive closure of $E_{\text{at}}$.

   \item  Given $(b, A, N)$, where $N \in \K$, $A \subseteq |N|$, and $b\in N$, the \emph{Galois type of $b$ over $A$ in $N$}, denoted by $\gtp (b/A, N)$, is the equivalence class of $(b, A, N)$ modulo $E$.

    \item For $M\in \K$, $\gS(M):=\{\type(a/M,N): M \lea N \text{ and }  a \in |N| \}$ denotes the set of all Galois types over $M$ and $\gS^{na}(M):=\{\type(a/M,N): M \lea N \text{ and }  a \in |N| \backslash |M| \}$ denotes the set of all \emph{non-algebraic types} over $M$.
    
    \item Given $p=\gtp(b/A, N)$ and $C \subseteq A$, let $p\upharpoonright{C}= [(b, C, N)]_{E}$. Given $M \lea N$, $p \in \gS(N)$ and $q \in  \gS(M)$, $p$ \emph{extends $q$}, denoted by $q \leq p$, if $p\restriction_M = q$. 
\end{enumerate}
\end{mydef}

Tameness and locality are properties asserting that distinct Galois types are witnessed locally. Tameness appears in some of the arguments of \cite{sh394} and was isolated in \cite{tamenessone}. Locality appears for the first time in-print in \cite{sh576}.

\begin{mydef}\label{tame}\
\begin{enumerate}
\item $\K$ is \emph{$(\kappa,\lambda)$-tame} if for every $M \in \K_\lambda$ and every $p,q\in \gS(M)$, if $p\neq q$, then there is $A \subseteq |M|$ of cardinality $\kappa$ such that $p\restriction_A \neq q\restriction_A$.
    \item $\K$ is \emph{$(< \kappa, \lambda)$-tame} if for every $M \in \K_\lambda$ and every $p,q\in \gS(M)$, if $p\neq q$, then there is $A \subseteq |M|$ of cardinality less than $\kappa$ such that $p\restriction_A \neq q\restriction_A$.
\end{enumerate}

\end{mydef}

\begin{mydef}\label{local}\
\begin{enumerate}
    \item $\K$ is \emph{$(\kappa,\lambda)$-local} if for every $M \in \K_\lambda$, every increasing continuous chain $\langle M_i : i<\kappa \rangle$ such that  $M=\bigcup_{i<\kappa}M_i$ and every $p,q\in \gS(M)$, if $p\restriction_{M_i}=q\restriction_{M_i}$ for all $i < \kappa$ then $p=q$.
    \item $\K$ is \emph{$(< \kappa, \lambda)$-local} if $\K$ is $(\mu,\lambda)$-local for all $\mu<\kappa$.
\end{enumerate}

\end{mydef}

Below are some relations between tameness and locality.
\begin{prop} Let $\lambda \geq \LS(\K)$.
\begin{enumerate}
\item If $\K$ is $(< \aleph_0, \lambda)$-tame, then $\K$ is $(< \lambda^+, \lambda)$-local.
\item Assume $\lambda > \LS(\K)$. If $\K$ is $( \lambda, \lambda)$-local, then $\K$ is $(<\lambda, \lambda)$-tame.
\item  If $\K$ is $( \mu , \mu)$-local for every $\mu \leq \lambda$, then $\K$ is $(\LS(\K), \mu)$-tame for every $\mu \leq \lambda$.
\item Assume $\lambda\geq \kappa$, $\cf (\kappa) >\chi$. If $\K$ is $(\chi,\lambda)$-tame, then $\K$ is  $(\kappa,\lambda)$-local.
\end{enumerate}
\end{prop}
\begin{proof}\
\begin{enumerate}
\item Straightforward.
\item Let $M \in \K_\lambda$ and $p, q \in \gS (M)$ such that $p\restriction_A = q\restriction_A$ for every $A \subseteq |M|$ with $|A| < \lambda$. Let $\langle M_i : i < \lambda \rangle$ be an increasing continuous chain such that $M=\bigcup_{i<\lambda}M_i$ and $\| M_i \| \leq \LS(\K) + |i|$ for every $i < \lambda$. Since  $\|M_i\| < \lambda$ for every $i < \lambda$, $p\restriction_{M_i} = q\restriction_{M_i}$ for every $i < \lambda$. Therefore, $p = q$ as $\K$ is $(\lambda, \lambda)$-local.
\item Similar to (2), see also \cite[1.18]{bales}.
\item This is  \cite[1.11]{bsh}
\end{enumerate} \end{proof}

\begin{remark} Universal classes are $(<\aleph_0, \lambda)$-tame for every $\lambda \geq \LS(\K)$ \cite[3.7]{vaseyd}, Quasiminimal AECs (in the sense of \cite{vaseyq}) are $(<\aleph_0, \lambda)$-tame for every $\lambda \geq \LS(\K)$ \cite[4.18]{vaseyq} and many natural AECs of modules are $(<\aleph_0, \lambda)$-tame for every $\lambda \geq \LS(\K)$ (see for example \cite[\S 3]{maz2}). The main results of this paper assume that the AEC is $(<\lambda^+, \lambda)$-local, so they apply to all of these classes. 

On the other hand there are AECs which are not $(\aleph_1, \aleph_1)$-local \cite{bsh} and which are not tame \cite{untame}.
\end{remark}

A natural question we were unable to answer is the following:
\begin{question}\label{q1}
If $\K$ is $( \aleph_0, \aleph_0)$-local, is $\K$ $(< \aleph_0, \aleph_0)$-tame?
\end{question}



Recall that if $\langle M_i : i<\omega\rangle$ is an increasing chain, then any increasing sequence of types $\langle p_i\in \gS(M_i) :  i<\omega\rangle$ has an upper bound. For longer sequences of types an upper bound might not exist.  A sufficient condition for a sequence to have an upper bounds is coherence (see for example \cite[3.14]{m1} for the definition). 

\begin{fact}\label{coh}
Let $\delta$ be a limit ordinal and $\langle M_i :  i\leq\delta\rangle$ be an increasing continuous chain. If $\langle p_i\in \gS^{na}(M_i) : i<\delta\rangle$ is a coherent sequence of types, then there is $p\in \gS^{na}(M_\delta)$ such that $p \geq p_i$ for every $i < \delta$ and $\langle p_i\in \gS^{na}(M_i) : i < \delta + 1\rangle$ is coherent. 
 \end{fact}

\section{Main results}

In this section we prove the main results of the paper. Throughout this section $\K$ is always an abstract elementary class and $\lambda$ is always a cardinal greater than or equal to the Löwenheim-Skolem-Tarski number of $\K$.  We begin by recalling the following notions that appear in  \cite{wy} and \cite[\S VI]{shelahaecbook}.

\begin{mydef}\
\begin{itemize}
\item $p=\type(a/M,N)$ has the \emph{$\lambda$-extension property} if for every $M'\in \K_\lambda$ $\lea$-extending $M$, there is $q\in \gS^{na}(M')$ extending $p$. In this case we say $p\in \gS^{\lambda-ext}(M)$.\footnote{These types are also called \emph{big types} in the literature, see for example \cite{sh48} and \cite{les}. }
\item  $p=\type(a/M,N)$ is \emph{$\lambda$-algebraic} if $p \in \gS^{na}(M)-\gS^{\lambda-ext}(M)$. Let   $\gS^{\lambda-al}(M)$ denote the $\lambda$-algebraic types over $M$.
\end{itemize}

\end{mydef}

Observe that if $p$ has the $\lambda$-extension property and $dom(p) \in \K_\lambda$ then $p$ is non-algebraic. 

Recall that an AEC $\K$ is \emph{stable in $\lambda$} if $| \gS(M)| \leq \lambda$ for every $M \in \K_\lambda$. We introduce a weakening of stability.

\begin{mydef}
    $\K$ is stable for $\lambda$-algebraic types in $\lambda$ if for all $M\in \K_\lambda$, $|\gS^{\lambda-al}(M)|\leq \lambda$.
\end{mydef}

Recall that an AEC $\K$ has disjoint amalgamation in $\lambda$  if for any $M, N_1, N_2 \in \K_\lambda$, if  $M \lea N_1, N_2$ and  $N_1 \cap N_2 = M$ then there are $N \in \K_\lambda$, $f_1: N_1 \to N$ and $f_2: N_2 \to N$ such that $f_1\rest_M= f_2\rest_M$ and $f_1[N_1] \cap f_2[N_2] = f_1[M](= f_2[M])$.  There are many AECs with disjoint amalgamation, see for example \cite[2.2]{baldwine},  \cite{disj}, \cite{disj2} and \cite[2.10]{mj}. 

\begin{prop}\label{p-disj}
If $\K$ has disjoint amalgamation in $\lambda$, then $\gS^{\lambda-al}(M) = \emptyset$ for every $M \in \K_\lambda$. In particular, $\K$ is stable for $\lambda$-algebraic types in $\lambda$.
\end{prop}
\begin{proof}
Let $M \in \K_\lambda$. Suppose $p=\type(a/M,N)$ for $a\in |N|-|M|$. Let $M' \in \K_\lambda$ with $M'\geq_\K M$. Let $M'' \in \K_\lambda$ and $f: M' \cong M''$ such that $M'' \cap N = M$, $M\lea M''$ and and $f\rest_M= id_M$. Amalgamate $M \lea M'', N$  such that for some $g$ the following commutes: 
\begin{equation*}
    \begin{tikzcd}
        M'' \ar{r}{id}& L\\
        M \ar{r}{id} \ar{u}{id}& N\ar{u}{g}.
    \end{tikzcd}
\end{equation*}

and $g[N]\cap M''=M$. This is possible by disjoint amalgamation in $\lambda$.

Let $L' \in \K_\lambda$ and  $h: L' \cong L$  such that $M' \lea L'$ and $h\rest_{M'} = f\rest_{M'}$. 

Let  $q= \type(h^{-1}(g(a))/M', L')$. It is straightforward to show that $q$ extends $p$ and $q$ is a non-algebraic type because $g(a) \notin |M''|$ as  $g[N]\cap M''=M$.
\end{proof}

\begin{remark}\label{unsta}
The previous result shows that stability for $\lambda$-algebraic types in $\lambda$ is strictly weaker than stability in $\lambda$ as  AECs axiomatizable by a complete first-order theory have disjoint amalgamation (in all cardinals).

\end{remark}

A natural question, suggested by the referee, we were unable to answer is the following:

\begin{question}\label{q-r}
Find an (natural) example of an AEC that is stable for $\lambda$-algebraic types in $\lambda$ but such that $\gS^{\lambda-al}(M) \neq \emptyset$ for some $M \in \K_\lambda$. 
\end{question} 
We begin working towards proving the main results of the paper.

\begin{lem} \label{many-r} Assume that $\K$ has amalgamation in $\lambda$,  no maximal model in $\lambda$ and is stable for $\lambda$-algebraic types in $\lambda$. Let $p \in \gS(M)$ and $M \in \K_\lambda$. The following are equivalent.
\begin{enumerate}
    \item $p$ has the $\lambda$-extension property. 
    \item $p$ has $\geq \lambda^+$ realizations in some $M' \in \K$ such that  $M \lea M'$.
    
\end{enumerate}
Moreover if $\K$ is stable for $\lambda$-algebraic types in $\lambda$, the conditions above are equivalent to:
\begin{enumerate}
  \setcounter{enumi}{2}
\item $p$ has an extension to a type with the $\lambda$-extension property for every  $M' \in \K_\lambda$ a $\lea$-extension  of $M$. 

\end{enumerate}
Moreover if $\K$ is stable in $\lambda$, the conditions above are equivalent to:
\begin{enumerate}
  \setcounter{enumi}{3}
  \item $p$ has an extension to a non-algebraic type to some $M'$ universal extension of $M$.\footnote{Recall that $M'$ is universal over $M$ if for any $N \in \K_{\lambda}$ such that
$M \leq_\K N$, there is $f: N \to M'$ a $\K$-embedding with $f\rest_M = id_M$.}
\end{enumerate}

\end{lem}
\begin{proof}
The equivalence between (1) and (2) appears in \cite[2.9]{sh576} (see also \cite[2.2]{bales}).

Assume $\K$  is stable for $\lambda$-algebraic types in $\lambda$. (3) implies (1) is clear, so we only need to show (2) implies (3). 

 Let $M' \in \K_\lambda$ with $M \lea M'$. By (2) there are  $N \geq_\K M$ and $\{a_i\in |N| :  i<\lambda^+\}$ distinct realizations of $p:= \type(a/M, N)$. Using amalgamation in $\lambda$ we may assume that $M' \lea N$. Moreover, we may assume without loss of generality that for all $i<\lambda^+$, $a_i\notin |M'|$. If not, subtract those $a_i$ that are in $M'$. Observe that $\type(a_i/M',N)\geq p$ and $\type(a_i/M',N) \in \gS^{na}(M')$ for all $i<\lambda^+$. If $|\{\type(a_i/M', N) :  i<\lambda^+\}|=\lambda^+$, it follows from stability for $\lambda$-algebraic types in $\lambda$ that for some $i < \lambda^+$, $\type(a_i/M', N) \in\gS^{\lambda-ext}(M)$.  Otherwise $| \{\type(a_i/M', N) :  i<\lambda^+\}|\leq \lambda$. Let $\Phi: \lambda^+ \to \{\type(a_i/M', N) :  i<\lambda^+\}$ be given by $i\mapsto \type(a_i/M',N)$. Since $|\{\type(a_i/M', N) :  i<\lambda^+\}|\leq \lambda$, by the pigeonhole principle there is $q\in \{\type(a_i/M', N) :  i<\lambda^+\}$ such that $|\{ i <\lambda^+: \Phi(i)=q\}| \geq \lambda^+$. That is, $q$ has $\lambda^+$-many realizations in $N$. Hence $q$ has the the $\lambda$-extension property by the equivalence between (1) and (2) for $q$.

Assume $\K$ is stable in $\lambda$. (1) implies (4) is clear as universal extensions exists by stability in $\lambda$, so we only need to show (4) implies (2).

Assume that $p:=\type(a/M,N)$ has an extension to $q = \type(b/M',N') \in \gS^{na}(M')$ for a universal extension $M'$ of $M$. We build $\langle M_i: i<\lambda^+\rangle $ increasing continuous such that $M_0 = M$ and for all $i$, there is $a_i\in |M_{i+1}|-|M_i|$ realizing $\type(a/M,N)$. If we can carry out this construction we are done as $\bigcup_{i < \lambda^+} M_i$ has $\lambda^+$ realizations of $\type(a/M,N)$. The base step and limit steps are clear so we only need to do the successor step. Since $M'$ is universal over $M$, there is $f:M_i \to M'$ with $f\rest_M = id_M$. Let $M_{i+1} \in \K_\lambda$ and  $g: M_{i+1} \cong N'$  such that $M_i \lea M_{i+1}$ and $g\rest_{M_i} = f\rest_{M_i}$. Let $a_{i+1} = g^{-1}(b)$. It is straightforward to show that $a_{i+1}$ is as required.  \end{proof}


We show that there are types with the $\lambda$-extension property.

\begin{lem}\label{ex-ext} Assume that $\K$ has amalgamation in $\lambda$ and no maximal model in $\lambda$. If  $\K$ is stable for $\lambda$-algebraic types in $\lambda$, then there is $p\in \gS^{\lambda-ext}(M)$ for every $M \in \K_\lambda$. 
\end{lem}
\begin{proof}
  Fix $M\in \K_\lambda$. There are two cases to consider. If $|\gS^{na}(M)|\geq \lambda^+$, the result  follows directly from the assumption that $\K$ is stable for $\lambda$-algebraic types in $\lambda$. If $|\gS^{na}(M)|\leq \lambda$, then a similar argument to that of (2) implies (3) of the previous lemma can be used to obtain the result. 
  \end{proof}

Recall the following notion. This notion was first introduced by Shelah in \cite[6.1]{sh48}, called minimal types there. Note that this is a different notion from the minimal types of \cite{sh576}. These types are also called \emph{quasiminimal types} in the literature, see for example \cite{les}.
 
\begin{mydef}\label{duni}
$p=\type(a/M,N)$ is a \emph{$\lambda$-unique type}
if
\begin{enumerate}
    \item $p=\type(a/M,N)$ has the $\lambda$-extension property. 
    \item For every  $M'\in \K_\lambda$ $\lea$-extending $M$, $p$ has at most one extension $q\in \gS^{\lambda-ext}(M')$.
\end{enumerate}

In this case we say that $p \in  \gS^{\lambda-unq}(M)$.
\end{mydef}

We show the existence of $\lambda$-unique types. The argument is standard (see for example \cite[2.15]{bales}), but we provide the details to show that the argument can be carried out in this setting.

\begin{lem}\label{l2} Assume that $\K$ has amalgamation in $\lambda$,  no maximal model in $\lambda$ and is stable for $\lambda$-algebraic types in $\lambda$.  If $|\gS^{na}(M)|<2^{\aleph_0}$ for every $M \in \K_\lambda$, then for every $M_0\in \K_\lambda$ and $p\in \gS^{\lambda-ext}(M_0)$, there is $M_1\in \K_\lambda$ and $q \in \gS^{\lambda-unq}(M_1)$ such that $M_0 \lea M_1$ and $q$ extends $p$.

\end{lem}
\begin{proof}
    Assume that $|\gS^{na}(M)|<2^{\aleph_0}$ for every $M \in \K_\lambda$ and assume for the sake of contradiction that the conclusion fails. Then there is $M_0\in \K_\lambda$ and $p\in \gS^{\lambda-ext}(M_0)$ without a $\lambda$-unique type above it.
    
    We build $\langle M_n : n <\omega\rangle$ and $\langle p_\eta: \eta\in 2^{<\omega} \rangle $ by induction such that:
    \begin{enumerate}
        \item $p_{\langle\rangle}=p$;
        \item for every $\eta \in 2^{<\omega}$, $p_\eta\in \gS^{\lambda-ext}(M_{\ell(\eta)})$;
        \item for every $\eta \in 2^{<\omega}$, $p_{\eta\concat 0}\neq p_{\eta\concat 1}$. \label{distinct_types}
        
    \end{enumerate}
    
    \fbox{Construction} The base step is given so we do the induction step. By induction hypothesis we have $\langle p_\eta \in \gS^{\lambda-ext}(M_n) : \eta \in 2^n \rangle$.  Since there is no $\lambda$-unique type above $p_{\langle\rangle}$ and by Lemma \ref{many-r}, for every $\eta  \in 2^n $ there are $N_\eta \in \K_\lambda$ and $q_\eta^0$, $q_\eta^1 \in \gS^{\lambda-ext}(N_\eta)$ such that $q_\eta^0$, $q_\eta^1 \geq p_\eta$ and $q_\eta^0 \neq q_\eta^1$. 
    
   Using amalgamation in $\lambda$  we build  $M_{n+1} \in \K_\lambda$ and $\langle f_\eta: N_\eta \xrightarrow[M_n]{} M_{n+1} : \eta \in 2^n \rangle$. Now for every $\eta \in 2^n$, let $p_{\eta \concat 0}, p_{\eta \concat 1} \in \gS^{\lambda-ext}(M_{n+1})$ such that $ p_{\eta \concat 0} \geq f_\eta(q_\eta^0)$ and $ p_{\eta \concat 1} \geq f_\eta(q_\eta^1)$. These exist by Lemma \ref{many-r}. It is easy to show that $M_{n+1}$ and $\langle p_{\eta\concat \ell} : \eta \in 2^n, \ell \in \{0, 1\} \rangle$  are as required.

\fbox{Enough} Let $N:=\bigcup_{n <\omega} M_n\in \K_\lambda$. For every $\eta \in  2^\omega$, let $p_\eta \in \gS^{na}(N)$ be an upper bound of $\langle p_{\eta\restriction_n} :  n<\omega\rangle$  given by Fact \ref{coh}. Observe that if $\eta\neq \nu\in 2^\omega $, $p_\eta\neq p_\nu$. Then $|\gS^{na}(N)|  \geq 2^{\aleph_0}$ which contradicts our assumption.
\end{proof}
\begin{remark}\label{r1}
    If $M\lea N$, $p\in \gS^{\lambda-unq}(M)$, $q \in \gS^{\lambda-ext}(N)$ and $q\geq p$, then $q\in \gS^{\lambda-unq}(N)$.
\end{remark}

We are ready to prove one of the main results of the paper.

\begin{thm}\label{main-2-weak}

Assume that $\K$ has amalgamation in $\lambda$,  no maximal model in $\lambda$, and is stable for $\lambda$-algebraic types in $\lambda$.  If $|\gS^{na}(M)|<2^{\aleph_0}$ for every $M \in \K_\lambda$ and $\K$ is $(<\lambda^+, \lambda)$-local, then $\K$ has a model of cardinality $\lambda^{++}$.
    
\end{thm}
\begin{proof}
It is enough to show that $\K$ has no maximal models in $\lambda^+$. 

Assume for the sake of contradiction that $M\in \K_{\lambda^+}$ is a maximal model. Let $N\lea M$ such that $N \in \K_\lambda$. By the maximality of $M$ together with Lemma \ref{ex-ext}, Lemma \ref{l2} and amalgamation in $\lambda$,  there is $M_0 \in \K_\lambda$ with $N \lea M_0 \lea M$ and  $q_0\in \gS^{\lambda-unq}(M_0)$. Let $\langle M_i\in \K_\lambda :  i<\lambda^+\rangle $ be a resolution of $M$ with $M_0$ as before. We build $\langle p_i :  i <\lambda^+\rangle$ such that:
\begin{enumerate}
    \item $p_0=q_0$;
    \item if $i< j< \lambda^+$, then $p_i\leq p_j$;
    \item for every $i < \lambda^+$,  $p_i\in \gS^{\lambda-unq}(M_i)$;
    \item for every $j < \lambda^+$, $\langle p_i :  i<j\rangle$ is coherent . 
\end{enumerate}
\fbox{Construction} The base step is given and the successor step can be achieved  using  Lemma \ref{many-r} and Remark \ref{r1}. So assume $i$ is limit, take $p_i$ to be an upper bound of $\langle p_j :  j<i\rangle$ given by Fact \ref{coh}. By Fact \ref{coh}  $\langle p_j :  j<i +1 \rangle$ is coherent so we only need to show that $p_i \in  \gS^{\lambda-unq}(\bigcup_{j<i}M_j )$.

By Remark \ref{r1} it suffices to show that $p_i \in \gS^{\lambda-ext}(\bigcup_{j<i}M_j)$. Since $p_0\in \gS^{\lambda-unq}(M_0)$ and $M_0\lea \bigcup_{j<i}M_j$, there is $q\in \gS^{\lambda-ext}(\bigcup_{j < i} M_j)$ such that $q\geq p_0$ by Lemma \ref{many-r}. 

We show that for every $j<i$, $q\restriction_{M_j}=p_i\restriction_{M_j}$. Let $j<i$. Since $q\restriction_{M_j}\in \gS^{\lambda-ext}(M_j)$, $p_i\restriction_{M_j}=p_j \in \gS^{\lambda-ext}(M_j)$ and  both extend $p_0$ a $\lambda$-unique type,  $q\restriction_{M_j}=p_i\restriction_{M_j}$.

Therefore, $q = p_i$ as $\K$ is $(<\lambda^+, \lambda)$-local. Hence $p_i \in \gS^{\lambda-ext}(\bigcup_{j<i}M_j)$ as $q\in \gS^{\lambda-ext}(\bigcup_{j < i} M_j)$.

 \fbox{Enough} Let $q^*\in \gS^{na}(M)$ be an upper bound of the coherent sequence $\langle p_i :  i <\lambda^+\rangle$ given by Fact \ref{coh}. As $q^*$ is a non-algebraic type, $M$ has a proper extension which contradicts our assumption that $M$ is maximal. \end{proof}
 
At the cost of strengthening the cardinal arithmetic hypothesis from $\lambda <2^{\aleph_0}$ to $\lambda^+ <2^{\aleph_0}$ we can drop the assumption that $|\gS^{na}(M)|<2^{\aleph_0}$ for every $M \in \K_\lambda$.

\begin{lem}\label{main-weak1} Suppose $\lambda^+<2^{\aleph_0}$.
 Assume $\K$ has amalgamation in $\lambda$, no maximal model in $\lambda$, and is stable for $\lambda$-algebraic types in $\lambda$. If $\K$ is $(<\lambda^+, \lambda)$-local, then $\K$ has a model of cardinality $\lambda^{++}$.
\end{lem}
\begin{proof}
 Assume for the sake of contradiction that $\K_{\lambda^{++}}=\emptyset$. We show that for every  $M\in \K_\lambda$, $|\gS^{na}(M)| < 2^{\aleph_0}$. This is enough by Theorem \ref{main-2-weak}.
 
Let $M\in \K_\lambda$. Then there is $M\lea N\in \K_{\lambda^+}$ maximal. Every $p\in \gS^{na}(M)$ is realized in $N$ by amalgamation in $\lambda$ and maximality of $N$. Thus $|\gS^{na}(M)|\leq \| N \| = \lambda^+$. Since $\lambda^+<2^{\aleph_0}$ by assumption, $|\gS^{na}(M)| < 2^{\aleph_0}$.
\end{proof}
\begin{remark}
Theorem \ref{main-2-weak} and Lemma \ref{main-weak1} are new even for universal classes with countable Löwenheim-Skolem-Tarski number and $\lambda = \aleph_0$.
\end{remark}

We use Theorem \ref{main-2-weak} to obtain the result mentioned in the abstract. 

\begin{thm}\label{main-weak2} Suppose $\lambda<2^{\aleph_0}$. Let $\K$ be an abstract elementary class with $\lambda \geq \LS(\K)$.  Assume $\K$ has amalgamation in $\lambda$, no maximal model in $\lambda$, and is stable in $\lambda$. If $\K$ is $(<\lambda^+, \lambda)$-local, then $\K$ has a model of cardinality $\lambda^{++}$.
\end{thm}
\begin{proof}
 We show that for every  $M\in \K_\lambda$, $|\gS^{na}(M)| < 2^{\aleph_0}$. This is enough by Theorem \ref{main-2-weak}. Let $M\in \K_\lambda$.  $|\gS^{na}(M)|\leq \lambda$ by stability in $\lambda$. Since $\lambda<2^{\aleph_0}$ by assumption, $|\gS^{na}(M)| < 2^{\aleph_0}$.
\end{proof}

\begin{remark}\label{re<} For AECs $\K$ with $\LS(\K) > \aleph_0$ or $\lambda  > \aleph_0$, the result is new even for universal classes. 
 For AECs $\K$ with $\LS(\K)= \aleph_0$ and $\lambda = \aleph_0$, the assumption that $\lambda < 2^{\aleph_0}$ is vacuous. This result for $(<\aleph_0, \aleph_0)$-tame AECs can be obtained using \cite[4.7]{shvas-apal}, \cite[5.8]{shvas-apal}, \cite[II.4.13]{shelahaecbook}, but the result has never been stated in the literature.  Moreover, the argument presented in this paper is significantly simpler than the argument using the results of Shelah and Vasey.  Furthermore,
 the result is new for $(\aleph_0, \aleph_0)$-local AECs.


 \end{remark}

\begin{remark} It is worth pointing out that  Theorem \ref{main-2-weak} is significantly stronger than Theorem \ref{main-weak2} as we only assume that the AEC is stable for $\lambda$-algebraic types in $\lambda$ instead of stable in $\lambda$. 
\end{remark}


\begin{thebibliography}{She01b}
 
  \bibitem[Bal09]{baldwinbook09}
John Baldwin, \textbf{Categoricity}, American Mathematical Society (2009).

\bibitem[BET07]{baldwine}
 John Baldwin, Paul Eklof, and Jan Trlifaj, \emph{$N^{\perp}$ as an abstract elementary class}, Annals of Pure and Applied Logic  \textbf{149} (2007), no. 1,25--39.
 
  \bibitem[BaLe06]{bales}
 John T. Baldwin and Olivier Lessmann, \emph{Uncountable categoricity of local abstract elementary classes with amalgamation}, Annals of Pure and Applied Logic,
\textbf{143} (2006), Issues 1–3, 29--42.


\bibitem[BaKo09]{untame}
John T. Baldwin and Alexei Kolesnikov, \emph{Categoricity, amalgamation, and
  tameness}, Israel Journal of Mathematics \textbf{170} (2009), no.~1,
411--443.
\bibitem[BKS09]{disj}
John T. Baldwin, Alexei Kolesnikov, and Saharon Shelah, \emph{The Amalgamation Spectrum}, The Journal of Symbolic Logic \textbf{74} (2009), no. 3,  914–-928.

\bibitem[BKL17]{disj2}
John T. Baldwin, Martin Koerwien, and Michael C. Laskowski, \emph{Disjoint Amalgamation in Locally Finite AEC},  The Journal of Symbolic Logic \textbf{82} (2017), no. 1, 98 -- 119.





\bibitem[BaSh08]{bsh}
 John T. Baldwin and Saharon Shelah, \emph{Examples of Non-locality},  Journal of Symbolic Logic
    Volume 73, Issue 3 (2008), 765--782.




\bibitem[Fri75]{friedman}
Harvey Friedman, \emph{Hundred and two problems in mathematical logic}, Journal of Symbolic Logic \textbf{41} (1975),
113--129.

\bibitem[Gro02]{grossberg2002}
Rami Grossberg, \emph{Classification theory for abstract elementary classes},
  Logic and Algebra (Yi~Zhang, ed.), vol. 302, American Mathematical Society,
  2002, 165--204.
  
  \bibitem[GrVan06]{tamenessone}
Rami Grossberg and Monica VanDieren, \emph{Galois-stability for tame abstract elementary classes}, Journal
  of Mathematical Logic \textbf{6} (2006), no.~1, 25--49.
  
 \bibitem[JaSh13]{jrsh875}
Adi Jarden and Saharon Shelah, \emph{Non-forking frames in abstract elementary classes}, Annals of Pure and Applied Logic \textbf{164} (2013), 135--191.

\bibitem[Leu23]{leu}
Samson Leung, \emph{Axiomatizing AECs and applications},
Annals of Pure and Applied Logic,
\textbf{174}(2023), Issue 5, 103248. 


\bibitem[Les05]{les}
Olivier Lessmann, \emph{ Upward Categoricity from a Successor Cardinal for Tame Abstract Classes with Amalgamation}, The Journal of Symbolic Logic, \textbf{70} (2005), no. 2, 639--660.


\bibitem[Maz20]{m1}
Marcos Mazari-Armida, \emph{Non-forking w-good frames},  Archive for Mathematical Logic \textbf{59} (2020), nos 1-2, 31--56.

\bibitem[Maz23]{maz2}
 Marcos Mazari-Armida, \emph{Some stable non-elementary classes of modules}, The Journal of Symbolic Logic, \textbf{88} (2023), no. 1, 93-- 117.
 
 \bibitem[MaRo]{mj}
Marcos Mazari-Armida and Jiri Rosicky, \emph{Relative injective modules, superstability and noetherian categories}, preprint, 25 pages, https://arxiv.org/abs/2308.02456

\bibitem[MaVa18]{mv}
Marcos Mazari-Armida and Sebastien Vasey, \emph{Universal classes near $\aleph_1$}, The Journal of Symbolic Logic \textbf{83} (2018), no. 4, 1633--1643. 

\bibitem[MaYa]{mawa}
Marcos Mazari-Armida and Wentao Yang, \emph{Building models in small cardinals in local abstract elementary classes},  arXiv:2310.14474v1. 

 \bibitem[Sh48]{sh48} Saharon Shelah, \emph{Categoricity in $\aleph _{1}$ of sentences in $L_{\omega _{1},\omega}(Q)$}, Israel Journal of  Mathematics  \textbf{20} (1975), 127--148. 

  \bibitem[Sh74]{sh74} Saharon Shelah, \emph{Appendix to: ``Models with second-order properties. II. Trees with no undefined branches}, Annals of Mathematical Logic  \textbf{2} (1975), Issue 2, 223-226.

\bibitem[Sh87a]{sh87a}
Saharon Shelah, \emph{Classification theory for nonelementary classes, I. The number of uncountable models of $\psi \in L_{\omega _{1},\omega }$. Part A},  Israel Journal of Mathematics \textbf{46} (1983), 212--240.

 \bibitem[Sh88]{Sh:88}
Saharon Shelah, \emph{Classification of nonelementary classes, {II}. {A}bstract
  elementary classes}, Classification theory (John Baldwin, ed.) (1987), 419--497.
  \bibitem[Sh300]{sh300}
Saharon Shelah, \emph{Universal classes}, Classification theory (John Baldwin, ed.) (1987), 264--418.

    \bibitem[Sh394]{sh394}
Saharon Shelah, \emph{Categoricity for abstract classes with amalgamation},
Annals of Pure and Applied Logic \textbf{98}(1999), Issues 1–3, 261--294.

  \bibitem[Sh576]{sh576}
Saharon Shelah, \emph{Categoricity of an abstract elementary class in two successive
  cardinals}, Israel Journal of Mathematics \textbf{126}(2001), 29--128.
  
   \bibitem[Sh:h]{shelahaecbook}
Saharon Shelah, \textbf{Classification Theory for Abstract Elementary Classes},
 vol. 1 \& 2, Mathematical Logic and Foundations, nos. 18 \& 20, College
  Publications (2009).
  
 \bibitem[ShVas18]{shvas-apal} Saharon Shelah and Sebastien Vasey, \emph{Abstract elementary classes stable in $\aleph_{0}$}, Annals of Pure and Applied Logic \textbf{169} (2018), no. 7, 565--587.
 \bibitem[Tar54]{tarski} Alfred Tarski, \emph{Contributions to the theory of models {I}}, Indagationes Mathematicae \textbf{16} (1954), 572--581.


   \bibitem[Vas16]{vaseya}
Sebastien Vasey, \emph{Building independence relations in Abstract Elementary Classes}, The Journal of Symbolic Logic \textbf{81} (2016), no. 1, 357--383. 

\bibitem[Vas17]{vaseyd}
Sebastien Vasey, \emph{Shelah's eventual categoricity conjecture in universal classes: part I}, Annals of Pure and Applied Logic \textbf{168} (2017), no. 9, 1609--1642. 

\bibitem[Vas18a]{vaseyq}
Sebastien Vasey, \emph{Quasiminimal abstract elementary classes}, Archive for Mathematical Logic \textbf{57} (2018), nos. 3-4, 299--315.


\bibitem[Vas18b]{vaseyt}
Sebastien Vasey, \emph{Toward a stability theory of tame abstract elementary classes}, Journal of Mathematical Logic \textbf{18} (2018), no. 2, 1850009.


 
 \bibitem[Vas22]{vaseys}
Sebastien Vasey, \emph{On categoricity in successive cardinals}, The Journal of Symbolic Logic \textbf{87} (2022), no. 2, 545--563.

 \bibitem[Yan24]{wyt}
Wentao Yang, \emph{On successive categoricity, stability and NIP in abstract elementary classes}, Ph. D. thesis, 2024.

 \bibitem[Yan]{wy}
Wentao Yang,  \emph{An NIP-like Notion in Abstract Elementary Classes}, Preprint. URL: https://arxiv.org/abs/2303.04125




\end{thebibliography}
\end{document}